\documentclass[12pt]{article}

\usepackage{amstext}
\usepackage{amssymb}
\usepackage{amsmath}
\usepackage{pb-diagram} \usepackage{amsthm}
\usepackage{rotating}
\usepackage{amsfonts}
\usepackage{graphicx}
\theoremstyle{plain}
 \newtheorem{thm}{Theorem}
 \newtheorem{lem}{Lemma}
 \newtheorem*{thm*}{Theorem}
 \newtheorem*{lem*}{Lemma}
 \newtheorem{cor}{Corollary}
 \newtheorem*{cor*}{Corollary}
 \newtheorem{prop}{Proposition}
 
\theoremstyle{definition}
 
 \newtheorem{ex}{Example}
  \newtheorem*{remark*}{Remark}
 \newtheorem{remark}{Remark}

\newcommand{\noi}{\noindent}

\begin{document}

\title{$2$-stratifold spines of closed $3$-manifolds}

\author{J. C. G\'{o}mez-Larra\~{n}aga\thanks{Centro de
Investigaci\'{o}n en Matem\'{a}ticas, A.P. 402, Guanajuato 36000, Gto. M\'{e}xico. jcarlos@cimat.mx} \and F.
Gonz\'alez-Acu\~na\thanks{Instituto de Matem\'aticas, UNAM, 62210 Cuernavaca, Morelos,
M\'{e}xico and Centro de
Investigaci\'{o}n en Matem\'{a}ticas, A.P. 402, Guanajuato 36000,
Gto. M\'{e}xico. fico@math.unam.mx} \and Wolfgang
Heil\thanks{Department of Mathematics, Florida State University,
Tallahasee, FL 32306, USA. heil@math.fsu.edu}}
\date{}

\maketitle

\begin{abstract} $2$-stratifolds are a generalization of $2$-manifolds in that there are disjoint simple closed branch curves. We obtain a list of all closed $3$-manifolds that have a $2$-stratifold as a spine. \footnote {\bf{{AMS classification numbers}}:  57N10, 57M20, 57M05} 
\end{abstract}

\section{Introduction}  

$2$-{\it stratifolds} form a special class of $2$-dimensional stratified spaces. A (closed with empty $0$-stratum) $2$-stratifold is a compact connected $2$-dimensional cell complex $X$ that contains a $1$-dimensional subcomplex $X^{(1)}$, consisting of branch curves, such that $X-X^{(1)}$ is a (not necessarily connected) $2$-manifold. The exact definition is given in section 2. $X$ can be constructed from a disjoint union $X^{(1)}$ of circles and compact $2$-manifolds $W^2$ by attaching each component of $\partial W^2$ to $X^{(1)}$ via a covering map $\psi: \partial W^2 \to X^{(1)}$, with $\psi^{-1} (x)>2$ for $x\in X^{(1)}$. A slightly more general class of $2$-dimensional stratified spaces, called {\it multibranched surfaces} and which have been defined and studied in \cite{MO}, is obtained by allowing boundary curves, i.e. considering a covering map $\psi:\partial W'\to X^{(1)}$, where $\partial W'$ is a sub collection of the components of $\partial W^2$.

$2$-stratifolds arise as the nerve of certain decompositions of $3$-manifolds into pieces where they determine whether the $\mathcal{G}$-category of the $3$-manifold is $2$ or $3$ (\cite{GGH}). They are related to {\it foams}, which include special spines of $3$-dimensional manifolds and which have been studied by Khovanov \cite{Ko} and Carter \cite{SC}. Simple $2$-dimensional stratified spaces arise in Topological Data Analysis \cite{B},  \cite{CL}. 

 Matsuzaki and Ozawa \cite{MO} show that $2$-stratifolds can be embedded in $\mathbb{R}^4$. Furthermore they show that they can be embedded into some orientable closed $3$-manifold if and only if their branch curves satisfy a certain regularity condition. However, the embeddings are not $\pi_1$-injective, i.e. the induced homomorphism of fundamental groups is not injective. In fact, there are many $2$-stratifolds whose fundamental group is not isomorphic to a subgroup of a $3$-manifold group; for example there are infinitely many $2$-stratifolds with (Baumslag-Solitar) non-Hopfian fundamental groups. These can not be embedded as $\pi_1$-injective subcomplexes into $3$-manifolds since $3$-manifold groups are residually finite. On the other hand, every $2$-manifold embeds $\pi_1$-injectively in some (Haken) $3$-manifold. Since subgroups of $3$-manifold groups are $3$-manifold groups, the following question arises:

Question 1. Which $3$-manifolds $M$ have fundamental groups isomorphic to the fundamental group of a $2$-stratifold?

The fundamental group of a closed $2$-manifold $S$ is isomorphic to the fundamental group of a closed $3$-manifold $M$ if and only if $S$ is the $2$-sphere or projective plane and $M$ is $S^3$ or $P^3$, respectively. Since $S^2$ is not a spine of $S^3$, the only closed $3$-manifold with a (closed) $2$-manifold spine is $P^3$. This motivates the next question:
 
 Question 2. Which closed $3$-manifolds $M$ have spines that are $2$-stratifolds?

The main results of this paper are Theorem \ref{piM} which answers question 1 for closed $3$-manifolds and Theorem \ref{mainthm}, which answers question 2 by showing that a closed $3$-manifold $M$ has a $2$-stratifold spine if and only if $M$ is a connected sum of lens spaces, $S^2$-bundles over $S^1$, and $P^2{\times}S^1$'s.  
 %%%%%%%%%%%%%%%%%%%%%%%%%%%%%%%%%%%%%%%%%%%%%%%%%

\section{$2$-stratifolds and their graphs.}

In this section we review the definitions of a $2$-stratifold $X$ and its associated graph $G_X$ given in \cite{GGH1}. 

A (closed) $2$-stratifold is a compact $2$-dimensional cell complex $X$ that contains a $1$-dimensional subcomplex $X^{(1)}$, such that $X-X^{(1)}$ is a  $2$-manifold ($X^{(1)}$ and $X-X^{(1)}$ need not be connected). A component $C\approx S^1$ of $X^1$ has a regular neighborhood $N(C)= N_{\pi}(C)$ that is homeomorphic to $(Y {\times}[0,1]) /(y,1)\sim (h(y),0)$, where $Y$ is the closed cone on the discrete space $\{1,2,...,d\}$ (for $d\geq 3$) and $h:Y\to Y$ is a homeomorphism whose restriction to $\{1,2,...,d\}$ is the permutation $\pi:\{1,2,...,d\}\to  \{1,2,...,d\}$. The space $N_{\pi}(C)$ depends only on the conjugacy class of $\pi \in S_d$ and therefore is determined by a partition of $d$. A component of $\partial N_{\pi}(C)$ corresponds then to a summand of the partition determined by $\pi$. Here the neighborhoods $N(C)$ are chosen sufficiently small so that for disjoint components $C$ and $C'$ of $X_1$, $N(C)$ is disjoint from $N(C' )$. 

Note that $X$ may also be described as a quotient space $W \cup_{\psi}X^{(1)}$, where $\psi:\partial W\to X^{(1)}$ is a covering map (and $|\psi^{-1}(x)| >2$ for every $x\in X^{(1)}$).

We construct an associated  bicolored graph $G=G_X$ of  $X=X_G$ by letting the white vertices $w$ of $G_X$ be the components $W$ of $M:=\overline{X-\cup_j N(C_j)}$ where $C_j$  runs over the components of $X^1$;  the black vertices $b_j$ are the $C_j$'s. An edge $e$ is a component $S$ of $\partial M$; it joins a white vertex $w$ corresponding to $W$ with a black vertex $b$ corresponding to $C_j $ if $S=W\cap N (C_j)$. The number of boundary components of $W$ is the number of adjacent edges of $W$. 

$G_X$ embeds naturally as a retract into $X_G$.

We label the white vertices $w$ with the  genus $g$ of $W$; here we use Neumann's \cite{N} convention of assining negative genus $g$ to nonorientable surfaces; for example the genus $g$ of the projective plane or the Moebius band is $-1$, the genus of the Klein bottle is $-2$. We orient all components $C_j$ and $S$ of $X^{(1)}$ and $\partial W$, resp., and assign a label $m$ to an edge $e$, where $|m|$ is the summand of the partition $\pi$ corresponding to the component $S\subset \partial N_{\pi}(C)$;  the sign of $m$ is determined by the orientation of $C_j$ and $S$. In terms of attaching maps, $m$ is the degree of the covering map $\psi:S\to C_j$ for the corresponding components of $\partial W$ and $X^{(1)}$.

  (Note that the partition $\pi$ of a black vertex is determined by the labels of its adjacent edges). 
%%%%%%%%%%%%%%%%%%%%%%%%%%%%%%%%%%%%%%%%%%%%%%%

\section{Structure of $\pi_1 (X_G)$}

In this section we obtain a natural presentation for the fundamental group of a $2$-stratifold $X_G$ with associated bicolored graph $G=G_X$ and describe $\pi_1 (X_G )$ as the fundamental group of a graph of groups $\mathcal{G}$ with the same underlying graph $G$.\\

\noindent For a given white vertex $w$, the compact $2$-manifold $W$ has conveniently oriented boundary curves $s_1 ,\dots , s_p$ such that\\

 \noi (*) \qquad $\pi_1 (W)=\langle s_1 ,\dots, s_p , y_1 ,\dots , y_n : s_1 \cdots s_p \cdot q =1 \rangle $\\
 
 \noindent  where $q=[y_1 ,y_2 ]\dots [y_{2g-1},y_{2g}]$,  if $W$ is orientable of genus $g$ and $n=2g$,
 
\noindent $q=y_1^2 \dots y_n^2$,  if $W$ is non-orientable of genus $-n$.\\

\noindent Let $\mathcal{B}$ be the set of black vertices, $\mathcal{W}$ the set of white vertices and choose a fixed maximal tree $T$ of $G$. Choose  orientations of the black vertices and of all boundary components of $M$ such that all labels of edges in $T$ are positive.\\

Then $\pi_1 (X_G )$ has a natural presentation with  \\

 \noindent \begin{tabular}{rl}
 generators: & $\{b \}_{{b }\in \mathcal{B}}$    \\
 
    & $\{s_1 ,\dots, s_p , y_1 ,\dots , y_n\}$, one set for each $w\in \mathcal{W}$,  as in $(*)$\\
   
   &  $\{t_i\}$, one $t_i$ for each  edge $c_i \in G-T$ between $w$ and $b$ \\
   
\end{tabular} \\

\noindent \begin{tabular}{rl}
and relations: & $s_1 \cdots s_p \cdot q =1$, one for each $w\in \mathcal{W}$,  as in $(*)$\\

&  $b^m =s_i$, for each edge $s_i \in T$ between $w$ and $b$ with label $m\geq 1$ \\

& $t_i^{-1}s_i t_i=b^{m_i}$, for each edge $s_i \in G-T$ between $w$ and $b$ with label $m_i \in \mathbb{Z}$. \\
\end{tabular} \\

\

As an example we show in Figure 1 (the graph of) a $2$-stratifold $X_G$ with $\pi_{1}(X_G )=\mathcal{F}$, an $F$-group as in Proposition (III)5.3 of \cite{LS}, with presentation\\

 \noi ($\mathcal{F}$) \qquad $\mathcal{F}=\langle c_1 ,\dots, c_p , y_1 ,\dots , y_n : c_1^{m_1}, \dots ,c_p^{m_p},\,c_1 \cdots c_p \cdot q =1 \rangle $
 
 where $p,n\geq 0$, all $m_i >1$ and $q=[y_1 ,y_2 ]\dots [y_{2g-1},y_{2g}]$ or $q=y_1^2 \dots y_n^2$.\\
 
 Here we have denoted the generators corresponding to the black vertices by $c_i$, rather than $b_i$, to indicate that the finite order elements correspond to attaching disks along the boundary curves of $W$.
 
\begin{figure}[ht]
\begin{center}
\includegraphics[width=2.5in]{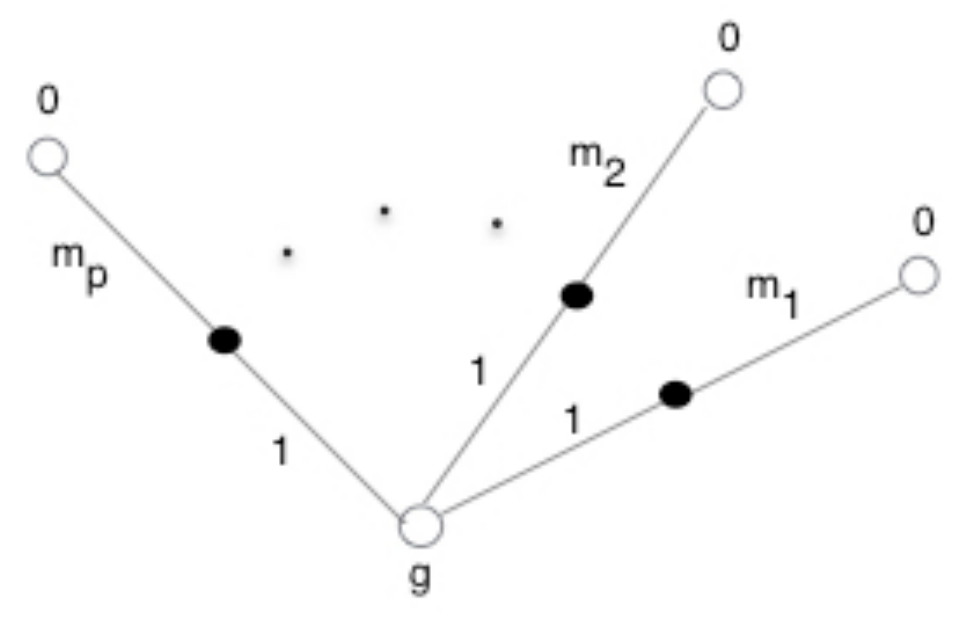} 
\end{center}
\caption{F-group}
\end{figure}

\vspace{2in}
The fundamental group of $X_G$ is best described as the fundamental group of a graph of groups  \cite{GGH2}. \\

If $\pi_1 (X_G )$ has no elements of finite order, then $\pi_1 (X_G )$ is the fundamental group of a graph of groups $\mathcal{G}$, with underlying graph $G$, the groups of white vertices are the fundamental groups of the $W's$, the groups of the black vertices and edges are (infinite) cyclic.

Elements of finite order occur when a generator $b$ of a black vertex has finite order $o(b)\geq 1$. In this case we attach $2$-cells $d_b$ and $d_e$ to $C_b$, the circle corresponding to $b$, as follows: $d_b$ is attached by a map of degree $o(b)$. If $e$ is an edge joining $b$ to $w$ with label $m$, attach $d_e$ with degree $o(c)=o(b)/(o(b),m)$. Letting $\hat{X}_b = N(C_b) \cup d_b \cup (\cup d_e )$, where $e$ runs over the edges having $b$ as an endpoint, $\hat{X}_w =W\cup (\cup d_e )$, where $e$ runs over the edges incident to $w$, and $\hat{X}_e = (\hat{X}_b \cap \hat{X}_w)$, for an edge $e$ joining $b$ to $w$, we obtain a graph of CW-complexes that determines a graph of groups $\mathcal{G}$ with the same underlying graph as $G_X$. 

The vertex groups are $G_b = \pi_1 (\hat{X}_b)$ and $G_w = \pi_1 (\hat{X}_w)$, the edge groups are $G_e = \pi_1 (\hat{X}_e)$, the monomorphisms $\delta: G_e \to G_b$ (resp. $G_e \to G_w$ are induced by inclusion. Then (see for example \cite{SW},\cite{Serre}) $\pi_1 \mathcal{G}\cong \pi_1 (\hat{X})$. 

Note that the groups $G_b$ of the black vertices and the groups $G_e$ of the edges are cyclic. For a white vertex $w$ with edges $e_1 ,\dots e_p$ labelled $m_1 ,\dots m_p$ with associated vertex space $X_w =W\cup_{i=1}^r d_{e_i}$ we obtain 

\noi $G_w =\langle c_1 ,\dots, c_p , y_1 ,\dots , y_n : c_1 \cdots c_p \cdot q =1\,,\,c_1^{k_1}=\dots =c_1^{k_r}= 1 \rangle$ 

\noi where $q$ is as in ($\mathcal{F}$), $1\leq r\leq p$ and $k_i \geq 1$.

If all $k_i \geq 2$ and $r=p$ then $G_w$ is an $F$-group (\cite{LS} p. 126-127), otherwise it is a free product of cyclic groups.

%%%%%%%%%%%%%%%%%%%%%%%%%%%%%%%%%%%%%%%%%%%%%

\section{Necessary Conditions}

In this section we show that a $2$-stratifold group that is a closed $3$-manifold group is a free product of cyclic or  $\mathbb{Z}{\times}\mathbb{Z}_2$ groups.\\

First consider an $F$-group $\mathcal{F}$ as in ($\mathcal{F}$). 

\begin{prop}\label{Fsubgrps} {\upshape(\cite{LS} Proposition (III)7.4)} Let $H$ be a subgroup of an $F$-group. If $H$ has finite index then $H$ is an $F$-group. If $H$ has infinite index then $H$ is a free product of cyclic groups.
\end{prop}

\begin{prop}\label{finiteF} {\upshape(\cite{LS} p.132)} (a) $\mathcal{F}$ is finite non-cyclic if and only if  $n=0,\, p=3$ and $(m_1 ,m_2 ,m_3 )=(2,2,m)$ ($m\geq 2$) (the dihedral group of order $2m$) or $(m_1 ,m_2 ,m_3 )=(2,3,k)$ for $k=3,4$ or $5$ (the tetrahedral, octahedral, dodecahedral groups).
In each case, $c_1$ is a non-central element of order $2$.

(b) $\mathcal{F}$ is finite cyclic if and only if  $n=0,\, p\leq 2$ (the $2$-sphere orbifold with at most two cone points)  or $n=1,\,p\leq 1$  (the projective plane orbifold with at most one cone point).
\end{prop}

\begin{lem} \label{product} $\mathcal{F}$ is not a non-trivial free product.
\end{lem}

\begin{proof} If $\mathcal{F}=A*B$ with $A,B$ non-trivial, then $A$ and $B$ have infinite index and so, by Proposition \ref{Fsubgrps}, $A,B$ and $\mathcal{F}$ are free products of cyclic groups. However, $\mathcal{F}$ is not such a group since it contains a subgroup isomorphic to the fundamental group of an orientable closed surface of genus $\geq 1$ (see the remark after Proposition (III)7.12 in \cite{LS}).
\end{proof}

The following remark is easy to see.

\begin{remark}\label{finiteorder} If $\mathcal{F}\neq \mathbb{Z}_2$ then $\mathcal{F}$ has no elements of finite order if and only if $\mathcal{F}$ is a surface group.
 \end{remark} 
 
\begin{lem}\label{surfacegroup} If $M$ is an orientable (not necessarily closed or compact) $3$-manifold with $\pi_1 (M)\cong \mathcal{F}$ then $\pi_1 (M)$ is cyclic or a surface group.
\end{lem}

\begin{proof}We may assume that $\partial M$ contains no $2$-spheres. By Scott's Core Theorem we may assume that $M$ is compact and by Lemma \ref{product}  that $M$ is irreducible.

If $\pi_1 (M)$ is infinite then $M$ is aspherical (see e.g. \cite{AFW}). It follows that  $\pi_1 (M)$ is torsion-free and from Remark \ref{finiteorder} that $\pi_1 (M)$ is a surface group.

If $\pi_1 (M)$ is finite then $M$ is closed. If $\pi_1 (M)$ is also non-cyclic then by Proposition \ref{finiteF}, $\pi_1 (M)$ contains a non-central element of order $2$. This can not happen by Milnor \cite{M}.
\end{proof}

We now consider a $2$-stratifold $X_G$ with $\pi_1 (X_G)=\pi(\mathcal{G})$ as in section 3.\\

Up to conjugacy, the only elements of finite order of $\pi_1 (X_G)$ are contained in the vertex groups; they correspond to black vertices of finite order and elements of white vertices $w$ whose corresponding group in $\mathcal{G}$ is finite. The latter are described in Proposition \ref{finiteF}. It is also shown in \cite{LS} (proof of Proposition (III)7.12) that in an infinite F-group the only elements of finite order are the obvious ones, namely conjugates of powers of $c_1,\dots ,c_p$.\\

For a group $H$, denote by $QH$ be the quotient group of $H$ modulo the smallest  subgroup of $H$ containing all elements of finite order of $H$.\\

Let $w$ be a white vertex in $G_X$. We say that $w$ is a {\it white hole}, if $w$ has label $-1$, all of its (black) neighbors have finite order and at most one of its neighbors has order $>1$.

If $G_X$ has more than one vertex, note that $Q\pi_1 (X_G)$ is obtained from $\pi_1 (X_G)$ by killing the open stars of all the black vertices representing elements of finite order $\geq 1$ of $\pi_1 (X_G)$ and deleting the white holes. In the example of Figure 1, when genus $g=-1$ (and so $n=1$), $Q\pi_1 (X_G)=\mathbb{Z}_2$. (Note that the white vertex of genus $-1$ is not a white hole if $p\geq 2, m_i >1$).

\begin{prop}\label{HX} If $Q(\pi_1 (X_G))$ has no elements of order $2$, then $H_3 (Q\pi_1 (X_G))=0$.
\end{prop} 

\begin{proof} Let $G'$ be the labelled subgraph of $G_X$ obtained  by deleting the open stars of all black vertices representing elements of finite order of $\pi_1 (X_{G})$ and all white holes.  ($\pi_1 (X_{\emptyset}) = 1$ by definition). Let $C$ be a component of $G'$. Then $Q\pi(X_{G}) = L* (*_C (\pi(X_C))$, the free product of a free group $L$ with the free product of the $\pi(X_C )$ where $C$ runs over the components of $G '$.

If $C$ consists of only one (white) vertex, then $X_{C}$ is a closed 2-manifold, different from $P^2$, since by assumption $Q(\pi_1 (X_G))$ has no elements of order $2$. We may ignore the $C$'s consisting of spheres, since they do not contribute to $Q\pi(X_{G})$. (A nonseparating $2$-sphere only changes the rank of $L$). In all other cases $X_C$ is the total space of a bicolored graph of spaces with white vertex spaces $2$-manifolds with boundary, edge spaces circles, and black vertex spaces homotopy equivalent to circles. 

Thus every vertex and edge space of $X_C$ is aspherical (with free fundamental group) of dimension $\leq 2$. By Proposition 3.6 (ii) of \cite{SW}, $X_C$ is aspherical. It follows that
$Q\pi_1 (X_{G})$ has (co)homological dimension $\leq 2$ and so $H_3 (Q\pi_1 (X_{G}))=0$.
\end{proof}

The assumption that $Q(\pi_1 (X_G))$ has no elements of finite order is satisfied if $\pi_1 (X_G)$ is a $3$-manifold group: We claim that $Q\pi_1 (M)$ is torsion free if $M$ is a closed orientable $3$-manifold.

For let $M=M_1 \# \dots \# M_k$ be its prime decomposition. If $M_i$ is irreducible with infinite fundamental group, then $M_i$ is aspherical and so $\pi_1 (M_i)$ is torsion free; if $M_i$ has finite fundamental group, then $Q\pi_1 (M_1)=1$. Now the claim follows since $Q\pi_1 (M)=Q\pi_1 (M_1 )*\dots *Q\pi_1 (M_k )$. 

 \begin{lem} \label{finite} Let $M$ be a closed orientable $3$-manifold with prime decomposition $M=M_1 \# \dots \# M_k$. If $\pi_1 (M)\cong \pi_1 (X_{G})$, then each $\pi_1 (M_i )$ is infinite cyclic or finite. 
\end{lem}
\begin{proof} If there is some $M_i$ with $\pi_1 (M_i )\neq \mathbb{Z}$, then $M_i$ is irreducible. If $\pi_1 (M_i )$ is infinite then $M_i$ is aspherical and hence $H_3 (Q\pi_1 (M_i ))=H_3 (\pi_1 (M_i ))=H_3 (M_i )\neq 0$. Since $Q\pi_1 (M)=Q\pi_1 (M_1 )*\dots *Q\pi_1 (M_k )$ it follows that $H_3 (Q\pi_1 (M ))\neq 0$, which contradicts Proposition \ref{HX}.
\end{proof}

\begin{lem} \label{finitecyclic} Let $M$ be a closed  $3$-manifold and suppose $\pi_1 (M) = \pi_1 (X_{G})$. Then any finite subgroup $H$ of $\pi_1 (M)$ is cyclic.
 \end{lem}
 
\begin{proof} $\pi_1 (X_{G}) \cong\pi_1 (\mathcal{G})$ where $\mathcal{G}$ is a graph of groups in which the groups of black vertices are cyclic and the groups of white vertices are $F$-groups or free products of finitely many cyclic groups. The finite group $H$ is non-splittable (i.e. not a non-trivial HHN extension or free product with amalgamation). By Corollary 3.8 and the Remark after Theorem 3.7 of \cite{SW}, $H$ is a cyclic group or isomorphic to a subgroup of an $F$-group. If $H$ is not cyclic, then (since $H$ is not a non-trivial free product of cyclic groups), $H$ is itself an $F$-group by Proposition \ref{Fsubgrps}. Since $H$ is a $3$-manifold group it follows from Lemma \ref{surfacegroup} that this case can not occur.
\end{proof}

\begin{cor} \label{orientable} Let $M$ be a closed orientable $3$-manifold. If $\pi_1 (M) \cong \pi_1 (X_{G})$, then $\pi_1 (M)$ is a free product of cyclic groups.
\end{cor}

\begin{proof} This follows from Lemmas \ref{finite} and \ref{finitecyclic}.
\end{proof}

\begin{thm} \label{piM}Let $M$ be a closed $3$-manifold. If $\pi_1 (M) \cong \pi_1 (X_{G}$), then $\pi_1 (M)$ is a free product of groups, where each factor is cyclic or $\mathbb{Z}{\times}\mathbb{Z}_2$.
\end{thm}
\begin{proof} If M is orientable this is Corollary \ref{orientable} (with each factor cyclic). Thus assume $M$ is non-orientable and let  $p:{\tilde M}\to M$ be the $2$-fold orientable cover of $M$. Then 
$\pi_1 ({\tilde M}) = \pi_1 (X_ {\tilde{G}})$ for the 2-stratifold $X_{\tilde{G}}$, which
is the 2-fold cover of $X_{G}$ corresponding to the orientation subgroup of $\pi_1 (M)$. Hence $\pi_1 ({\tilde M})$ is a free product of cyclic groups.

Let $M={\hat M_1} \# \dots \# {\hat M_k}$ be a prime decomposition of $M=M_1 \cup \dots \cup M_k$. If $M_i$ is orientable, then $M_i$ lifts to two homeomorphic copies ${\tilde M}_{i1}$, ${\tilde M}_{i2} $ of $M_i$, with each ${\hat{\tilde M}}_{ij}$ a factor of the prime decomposition of ${\tilde M}$ and it follows that $\pi_1 (M_i )$ is cyclic.

If ${\hat M_i}$ is non-orientable and $P^2$-irreducible, then $M_i$ lifts to  ${\tilde M}_i$, where ${\hat{\tilde M}}_i$ is irreducible. Then $\pi_1 ({\hat{\tilde M}}_i )$, being a factor of the free product decomposition of $\pi_1 ({\tilde M})$, is finite cyclic, which can not occur since $\pi_1 (M_i )$ is infinite.

If ${\hat M}_i$ is non-orientable irreducible, contains $P^2$'s, but is not  $P^2{\times}S^1$, then by Proposition (2.2) of \cite{S}, $M_i$ splits along two-sided $P^2$'s into $3$-manifolds $N_1 ,\dots ,N_m$ such that the fundamental group of the lifts ${\tilde N_i}$ is indecomposable, torsion free and not isomorphic to $\mathbb{Z}$. Since $\pi_1 ({\tilde N_i})$ is a factor of the free product decomposition of $\pi_1 ({\tilde M})$, this can not happen.

Therefore each non-orientable $M_i$ is either the $S^2$-bundle over $S^1$ or $P^2{\times}S^1$, which proves the Theorem.
\end{proof}
%%%%%%%%%%%%%%%%%%%%%%%%%%%%%%%%%%%%%%%%%%%%%
\section{Realizations of spines.}

Recall that a subpolyhedron $P$ of a $3$-manifold $M$ is a {\it spine} of $M$, if $M-Int(B^3 )$ collapses to $P$, where $B^3$ is a $3$-ball in $M$. 

An equivalent definition is that $M-P$ is homeomorphic to an open $3$-ball (Theorem 1.1.7 of \cite{Ma}).\\

We first construct $2$-stratifold spines of lens spaces (different from $S^3$), the non-orientable $S^2$-bundle over $S^1$, and $P^2{\times}S^1$.\\

\begin{ex} Lens space $L(0,1)=S^3$.

$S^3$ does not have a $2$-stratifold spine. Otherwise such a spine $X$ would be a deformation retract of the $3$-ball and therefore contractible. However there are no contractible 2-stratifolds \cite{GGH1}. \end{ex}

\begin{ex} Lens spaces $L(p,q)$ with $q\neq 0, 1$.

 Let $r$ be the rotation of the disk $D^2$ about its center $c$ with angle $2\pi p/q$, let $1\in S^1 \subset D^2$ and let $x_i =r^{i-1} (1)$, $i=1, \dots , q$. Let $Y\subset D^2$ be the cone of $\{x_1, \dots ,x_q\}$ with cone point $c$. Embed $Y{\times}I/(x_i ,0)\sim (x_{i+1},1)$ into the solid torus $V=D^2{\times}I/(x,0)\sim (r(x),1)$. The punctured lens space $L(p,q)$ is obtained from $V$ by attaching a $2$-handle $D{\times}I$ with $\partial D$ attached to the boundary curve of $(Y{\times}I)/_{\sim}$. Then $L(p,q)$ deformation retracts to $(Y{\times}I)/_{\sim} \cup D$, which is the $2$-stratifold with one white vertex of genus $0$, one black vertex, and one edge with label $q$.\end{ex}
 
\begin{ex} Lens space $L(1,0)=S^2{\times}S^1$ and non-orientable $S^2$-bundle over $S^1$.
 
 Consider $S^2{\tilde{\times}}S^1$, the non-orientable $S^2$-bundle over $S^1$, as as the quotient space $q(S^2{\times}I)$ under the quotient map $q:S^2{\times}I\to S^2{\times}S^1$ that identifies $(x,0)$ with $(x,1)$, $x\in S^2$.
 
 Let $D_0 \subset S^2{\times}\{0\}$ be a disk and $B_1$ be the $3$-ball $D_0{\times}I \subset S^2{\times}I$, let $D_1$ be the disk $B_1 \cap S^2{\times}\{1\}$, let $A$ be the annulus $\partial B_1 -(Int(D_0 ) \cup Int(D_1 ))$, and let $B_2$ be the ball $S^2{\times}I-Int(B_1 )$, see Figure 2. 
 
 \begin{figure}[ht]
\begin{center}
\includegraphics[width=4in]{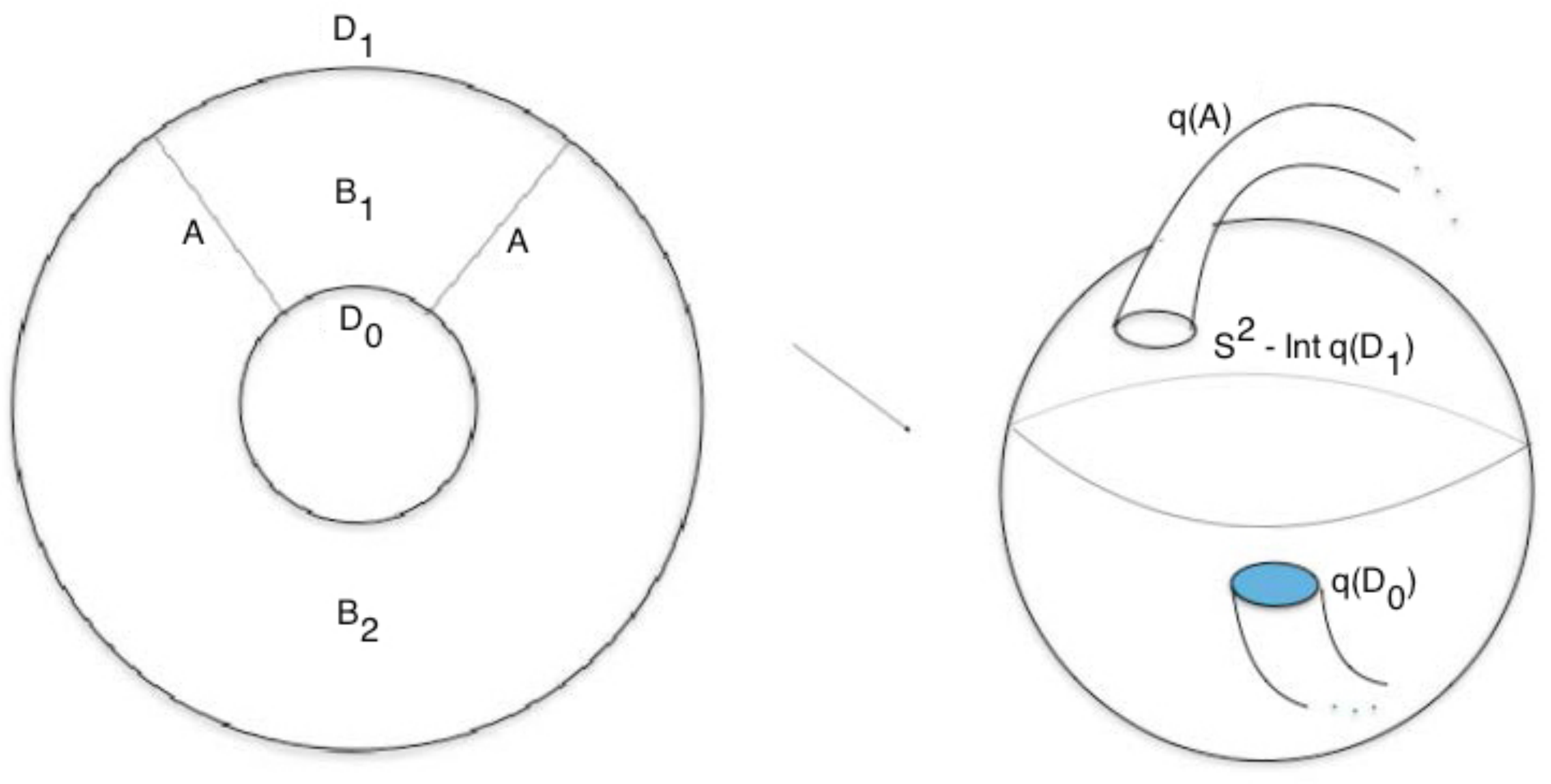} 
\end{center}
\caption{$S^2{\tilde{\times}}S^1 -Int(B_2 )\searrow S^1{\tilde{\times}}S^1\cup D^2 $}
\end{figure}
 Then $S^2{\times}I-Int(B_2 )=S^2{\times}\{0\}\cup B_1 \cup S^2{\times}\{1\}$ and $S^2{\tilde{\times}}S^1 -Int(B_2 )=q(S^2{\times}I-Int(B_2 ))=S^2 \cup q(B_1 )$, where $S^2 =q(S^2{\times}\{0\})=q(S^2{\times}\{1\})$. Collapsing the ball $q(B_1 )$ across the free face $q(D_1 )$ onto $q(A)\cup q(D_0 )$ we obtain a collapse of $S^2{\tilde{\times}}S^1 -Int(B_2 )$ onto $(S^2 -Int(q(D_1 )))\cup q(A)$, which is a Kleinbottle with a disk attached.  This is a $2$-stratifold $X_G$  with graph $G_X$ in Figure 3(a). (The white vertices have genus $0$).\\

 A similar construction, considering $S^2{\times}S^1$ as the obvious quotient space of $q:S^2{\times}I\to S^2{\times}S^1$ and first isotoping the ball $B_1$ such that $D_0 \cap D_1 = \emptyset $, we obtain a collapse of $S^2{\times}S^1 -Int(B_2 )$ onto a torus with a disk attached. This is a $2$-stratifold $X_G$  with graph $G_X$ in Figure 3(b). \end{ex}
  \begin{figure}[h]
\begin{center}
\includegraphics[width=4in]{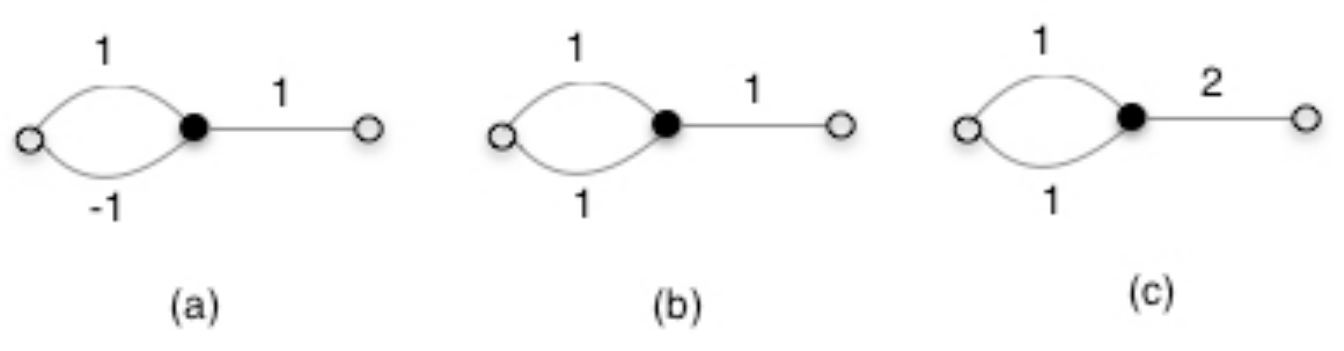} 
\end{center}
\caption{$2$-stratifold spines of punctured $S^2{\tilde\times}S^1$ and $P^2{\times}S^1$ }
\end{figure}

\begin{ex} $P^2{\times}S^1$.
  
For a one-sided simple closed curve $c$ in $P^2$ and a point $t_0$  in $S^1$  let $X=P^2{\times}\{t_0\}\cup c{\times}S^1 \subset  P^2{\times}S^1$. Observe that the boundary of a regular neighborhood $N$ of $X$ in $P^2{\times}S^1$ is a $2$-sphere. Since $P^2{\times}S^1$ is irreducible, $\partial N$ bounds a $3$-ball $B^3$ and therefore $P^2{\times}S^1 -Int(B^3 )=N$, which collapses onto $X=X_G$, a $2$-stratifold  with graph in Figure 3(c). \end{ex}

\begin{prop}\label{disksum} If the closed $3$-manifold $M_i \,(i=1,2)$ has a $2$-stratifold spine and $M$ is a connected sum of $M_1$ and $M_2$, then $M$ has a $2$-stratifold spine.
\end{prop}

\begin{proof} Let $K_i$ be a $2$-stratifold spine of $M_i$. Let $K_1\vee K_2$ be obtained by identifying, in the disjoint union of $K_1$ and $K_2$ a nonsingular point of $K_1$ with a nonsingular point of $K_2$. By Lemma 1 of \cite{K}, $K_1\vee K_2$  is a spine of $M$. Though $K_1\vee K_2$ is not a $2$-stratifold, by performing the operation explained below (replacing the wedge point by a disk) we will change $K_1\vee K_2$ to a $2$-stratifold spine $K_1\Delta K_2$. 

A $3$-ball neighborhood $B^3$ of the wedge point of $K_1\vee K_2$  intersects $K_1\vee K_2$ in the double cone shown in Fig.4. Replace, in $K_1\vee K_2$,   $K_1\vee K_2\cap B^3$ by $A\cup D$, as shown in Fig. 2, where $A=S^1{\times}[0,1]$  is an horizontal cylinder, $\partial A= (K_1\vee K_2)\cap \partial B^3$, $D$ is a vertical $2$-disk with $A\cap D=\partial D=S^1{\times}(1/2)$. The result is a $2$-stratifold $K_1\Delta K_2$. There is  a homeomorphism from $B^3 -(A\cup D)$ onto $B^3 - K_1\vee K_2$ which is the identity on the boundary (roughly collapse $D$ to a point) and so 
$M - K_1\Delta K_2$  is homeomorphic to $M - K_1\vee K_2$ which is homeomorphic to $R^3$.

Therefore $K_1\vee K_2$ is a $2$-stratifold spine of $M$.
\end{proof}
 \begin{figure}[h]
\begin{center}
\includegraphics[width=3in]{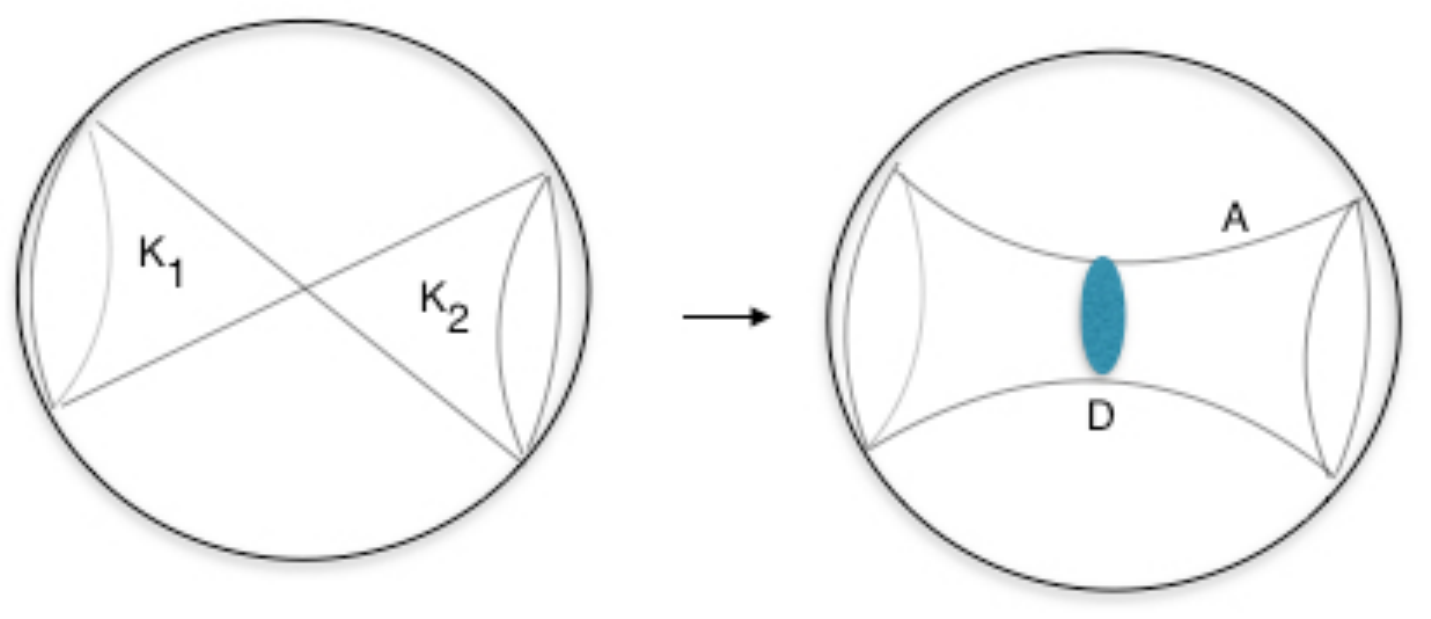} 
\end{center}
\caption{$K_1\Delta K_2$}
\end{figure}
Now Theorem \ref{piM} together with the examples and Proposition \ref{disksum} yields our main Theorem. Here we do not consider $S^3$ to be a lens space.

\begin{thm} \label{mainthm} A closed $3$-manifold $M$ has a $2$-stratifold as a spine if and only if $M$ is a connected sum of lens spaces, $S^2$-bundles over $S^1$, and $P^2{\times}S^1$'s. 
\end{thm}

 \
 
 {\it Acknowledgments:} J. C. G\'{o}mez-Larra\~{n}aga would like to thank LAISLA and the TDA project from CIMAT for  financial support and IST Austria for their hospitality.


\begin{thebibliography}{99}

\bibitem {AFW} M. Aschenbrenner, S. Friedel, H. Wilson, Decision Problems for $3$-manifolds and their fundamental groups, arXiv:1405.6274v2 [math.GT], (2015).

\bibitem {B} P. Bendich, E. Gasparovicy, C.J. Traliez, J. Harer, Scaffoldings and Spines: Organizing High-Dimensional Data Using Cover Trees, Local Principal Component Analysis, and Persistent Homology, arXiv:1602.06245v2 [cs.CG] 27 Feb 2016.

\bibitem {SC} J.S. Carter, Reidemeister/Roseman-type moves to embedded foams in 4-dimensional space. arXiv:1210.3608v1 [math.GT] 

\bibitem {GGH} J.C. G\'{o}mez-Larra\~{n}aga, F. Gonz\'alez-Acu\~na, Wolfgang Heil,  Categorical group invariants of 3-manifolds,  manuscripta math. 145 (2014), 433-448.

\bibitem {GGH1} J.C. G\'{o}mez-Larra\~{n}aga, F. Gonz\'alez-Acu\~na, Wolfgang Heil, $2$-stratifolds, in ``A Mathematical Tribute to Jos\'e Mar\'ia Montesinos Amilibia", Universidad Complutense de Madrid, 395-405 (2016).

\bibitem {GGH2} J.C. G\'{o}mez-Larra\~{n}aga, F. Gonz\'alez-Acu\~na, Wolfgang Heil, $2$-stratifold groups have solvable Word Problem, arXiv:1704.00686 [math.GT] (1917).

\bibitem{K} G.W. Knutson, A chacterization of closed 3-manifolds with spines containing no wild arcs,  Proc.Amer.Math.Soc. 21, 310-114 (1969).

\bibitem {Ko} M. Khovanov,  sl(3) link homology. Algebr. and Geom. Topol. 4, 1045-1081 (2004).

\bibitem {CL} P. Lum, G. Singh, J. Carlsson, A. Lehman, T. Ishkhanov, M. Vejdemo-Johansson, M. Alagappan, G. Carlsson, Extracting insights from the shape of complex data using topology. Nature Scientific Reports 3, 12-36 (2013).

\bibitem {LS} R. C. Lyndon and P. E. Schupp, Combinatorial Group Theory, Modern Surveys in Math., no. 89, Springer Verlag, Berlin, 1977.

\bibitem{MO} S. Matsuzaki and M. Ozawa, Genera and minors of multibranched surfaces, arXiv:1603.09041v1 [math.GT] 30 Mar 2016.

\bibitem{Ma} Algorithmic Topology and Classification of 3-Manifolds, S. Matveev, in Algorithms and Computation in Mathematics Vol 9, Springer (2007).

\bibitem {M} J. Milnor Groups which act without fixed points on $S^n$, Amer. J. Math. 79  623-630 (1957).

\bibitem {N}  W. Neumann, A calculus for plumbing applied to the topology of complex surface singularities and degenerating complex curves, Trans. Amer. Math. Soc. 268, 299-344 (1981).

\bibitem {SW} P. Scott and C.T.C.Wall, Topological Methods in Group Theory, In Homological Group Theory, London Math. Soc. Lecture Notes Ser. 36, Cambridge Univ. Press (1979).

\bibitem{Serre} J.P. Serre, Trees, Springer-Verlag, 1980. 

\bibitem{S} G.A.Swarup, Projective planes in irreducible $3$-manifolds, Math.Z. 132, 305-317 (1973).

\end{thebibliography}
\end{document}